\documentclass[12pt,letterpaper]{amsart}
\usepackage{geometry,graphicx}
\usepackage{setspace}
\usepackage{mathpazo} 
\usepackage[colorlinks,allcolors=blue]{hyperref} 
\usepackage{xpatch}
\xpatchcmd{\proof}
{\itshape}
{\bfseries}
{}
{}
\newtheorem{theorem}{Theorem}
\newtheorem{proposition}{Proposition}
\newtheorem{lemma}{Lemma}
\newtheorem{corollary}{Corollary}
\newtheorem{definition}{Definition}
\theoremstyle{remark}
\newtheorem{remark}{Remark}

\newcommand{\C}{\mathbb{C}}
\newcommand{\D}{\Omega}

\newcommand{\ep}{\varepsilon}
\newcommand{\Ha}{\mathbb{H}}

\newcommand{\zb}{\overline{z}}

\makeatletter
\@namedef{subjclassname@2020}{%
\textup{2020} Mathematics Subject Classification}
\makeatother

\onehalfspace

\title{A sufficient condition for $L^p$ regularity of the Berezin transform} 

\author{N\.{i}hat G\"{o}khan G\"{o}\u{g}\"{u}\c{s}}
\address[N\.{i}hat G\"{o}khan G\"{o}\u{g}\"{u}\c{s}]{Sabanc\i{} University,
Tuzla, 34956, Istanbul, Turkey}
\email{nggogus@sabanciuniv.edu}

\author{S\"{o}nmez \c{S}ahuto\u{g}lu}
\address[S\"{o}nmez \c{S}ahuto\u{g}lu]{University of Toledo, Department of
Mathematics \& Statistics, Toledo, OH 43606, USA}
\email{sonmez.sahutoglu@utoledo.edu}


\date{\today}
\begin{document}

\begin{abstract}
We prove that the Berezin transform is $L^p$ regular on a large class of domains 
in $\mathbb{C}^n$ and not $L^2$ regular on the Hartogs triangle. 
\end{abstract}

\maketitle

Let $\D$ be a domain in $\C^n$ with boundary $b\D$. We denote the 
Bergman space by $A^2(\D)$ and the Bergman projection, the orthogonal 
projection, by $P_{\D}:L^2(\D)\to A^2(\D)$. The Bergman Kernel of $\D$ 
is denoted by $K_{\D}:\D\times \D\to \C$ and, in case $K_{\D}(z,z)\neq 0$ for all $z\in \D$, 
\[k_z^{\D}(w)=\frac{K_{\D}(w,z)}{\sqrt{K_{\D}(z,z)}}\] 
is called the normalized Bergman kernel as $\|k_z^{\D}\|_{L^2}=1$ for all $z\in \D$ 
where $\|.\|_{L^p}$ denotes the $L^p$ norm on $\D$. 
In this paper, $K_{\D}(z,z)\neq 0$ for all $z\in \D$ will be a running assumption. 

We note that Boas in \cite{Boas86} observed that the monomial $z_1^jz_2^k$ is square-integrable 
on  the unbounded logarithmically convex Reinhardt domain 
$\mathcal{B}=\{(z_1,z_2)\in \C^2:|z_2|<(1+|z_1|)^{-1}\}$ if and only if $0 \leq j < k$. 
Hence  $K_{\mathcal{B}}(0,0)=0$. Later Engli\v{s} in \cite{Englis07} constructed the unbounded 
pseudoconvex complete Reinhardt domain 
$\mathcal{E}=\{(z_1,z_2,z_3)\in \C^3:|z_3|^2(|z_1|^2+|z_2|^2)^2<1,|z_1|^2+|z_2|^2<1\}$ such that  
$A^2(\mathcal{E})$ is infinite dimensional, yet $|K_{\mathcal{E}}(w, z)|^2 /K_{\mathcal{E}}(z,z)$ 
is discontinuous at $z=0$ for each fixed $w\neq 0$. Moreover, he showed that the Berezin 
transform (defined below) of $|z_1|^2$, a smooth bounded function on $\mathcal{E}$, is 
discontinuous at the origin. 

Next we will define the Berezin transform of a bounded operator $T:A^2(\D)\to A^2(\D)$. 
The Berezin transform $B_{\D}T(z)$ of $T$ at $z\in \D$ is defined as 
\[B_{\D}T(z)=\langle Tk^{\D}_z,k^{\D}_z\rangle\]
where  $\langle .,.\rangle$ denotes the inner product on $L^2(\D)$. 
For a function $\phi\in L^{\infty}(\D)$ we define its Berezin transform as 
\[B_{\D}\phi(z)=\langle T_{\phi}k^{\D}_z,k^{\D}_z\rangle\]
where $T_{\phi}:A^2(\D)\to A^2(\D)$ is the Toeplitz operator with symbol $\phi$. 
Namely, $T_{\phi}f=P_{\D}(\phi f)$.  We note that 
$B_{\D}\phi(z)=\int \phi(w)|k^{\D}_z(w)|^2dV(w)$ for $\phi\in L^{\infty}(\D)$. 

The Berezin transform has been an important tool in operator theory. For instance, 
the Axler-Zheng theorem and its various extensions establish compactness of operators 
using boundary values of the Berezin transform (see, for instance, 
\cite{AxlerZheng98,Suarez07,Englis99,CuckovicSahutoglu13,CuckovicSahutogluZeytuncu18}).  
The Berezin transform is also related to Schatten norm of Hankel (and Toeplitz) operators (see 
\cite{Zhu91,Li93,GogusSahutoglu18}). We refer the reader to \cite{ZhuBook} for more 
details about the Berezin transform.   

The fact that $\|k_z^{\D}\|_{L^2}=1$ for all $z\in \D$ implies that $\|B_{\D}\|_{L^{\infty}}\leq 1$ 
on any domain $\D$ with the property that $K_{\D}(z,z)\neq 0$ for all $z\in \D$. Furthermore, 
$\|B_{\D}\|_{L^{\infty}}=1$ if $\D$ is bounded. Even though $B_{\D}$ is defined on $L^{\infty}(\D)$, 
on some domains, such as the unit ball and the polydisc, it extends as a bounded map on 
$L^p(\D)$ for some $p<\infty$.  In this article, we want to know on what domains this is the case.  
This problem has been studied in few cases. Dostani\'c  in \cite{Dostanic08} showed that the 
Berezin transform is $L^p$ regular on the unit disc and computed the norm 
$\|B_{\mathbb{D}}\|_{L^p}=\pi(p+1)/(p^2\sin(\pi/p))$ for $1<p\leq \infty$. 
Later Liu and Zhou in \cite{LiuZhou12} and Markovi\'c in \cite{Markovic15} generalized 
Dostani\'c's result to the unit ball. On the bidisc $\mathbb{D}^2$, Lee in \cite[Proposition 3.2]{Lee97} 
showed that the Berezin transform is bounded on $L^p(\mathbb{D}^2)$ for $p>1$ and 
unbounded on $L^1(\mathbb{D}^2)$. In this article we are interested in investigating $L^p$ 
regularity of the Berezin transform on domains in $\C^n$. However, we will 
not attempt to compute its norm.

Now we define the class of domains on which we will prove $L^p$ regularity of the 
Berezin transform.  

\begin{definition}
Let $\D$ and $K_{\D}$ be a  domain in $\C^n$ and the Bergman kernel of $\D$, 
respectively. We say that $\D$ satisfies property BR if $K_{\D}(z,z)\neq 0$ for all $z\in \D$ and 
\[\sup\left\{\frac{|K_{\D}(w,z)|}{K_{\D}(z,z)}:z,w\in \D\right\}<\infty.\]
\end{definition}

The infimum of $\left\{\frac{|K_{\D}(w,z)|}{K_{\D}(z,z)}:z,w\in \D\right\}$ is 0  
whenever $\D$ is a bounded domain in $\C^n$. One can see this as follows: 
using a holomorphic affine change of coordinates, if necessary,  we assume that 
$\D\subset \{z\in \C^n:\text{Re}(z_n)>0\}$ and $0\in b\D$. Next we define 
$f(z)=\log(z_n)$. Then $|f(z)|\to\infty$ as $z\to 0$ and one can show that $f\in A^2(\D)$. 
That is, $K_{\D}(z,z)\to \infty$ as $z\to 0$. Hence for a fixed $w\in \D$ we have 
\[\frac{|K_{\D}(w,z)|}{K_{\D}(z,z)}\leq \frac{\sqrt{K_{\D}(z,z)}\sqrt{K_{\D}(w,w)}}{K_{\D}(z,z)}
= \frac{\sqrt{K_{\D}(w,w)}}{\sqrt{K_{\D}(z,z)}}\to 0 \text{ as } z\to 0.\]
Therefore,  $\inf\left\{\frac{|K_{\D}(w,z)|}{K_{\D}(z,z)}:z,w\in \D\right\}=0$.

Property BR is a biholomorphic invariant in the following sense: 
Assume that $\D_1$ satisfies property BR and $F:\D_1\to \D_2$ is a biholomorphism 
with the property that there exists $1<C<\infty$ such that $C^{-1}<|\det J_{\C}F(z)|<C$ for all $z\in \D_1$ 
where $J_{\C}F$ denotes the holomorphic Jacobian matrix of $F$. Then $\D_2$ satisfies 
property BR. One can easily see this using the  transformation formula for the 
Bergman kernel (see, for instance, \cite[Proposition 1.4.12]{KrantzBook}, 
\cite[Proposition 12.1.10]{JarnickiPflugBook1Ed2}, \cite[Theorem 4.9]{RangeBook}).  

Domains with property BR are not necessarily bounded or pseudoconvex. 
For example, one can show that the upper half plane in $\mathbb{C}$ 
and $\{z\in \C^2:1<|z|<2\}$, a bounded non-pseudoconvex domain, have property BR. 
On the other hand, not every pseudoconvex domain satisfies property BR. 
For instance, the Hartogs triangle, 
\[\mathbb{H}=\{(z_1,z_2)\in\C^2:|z_2|<|z_1|<1\},\]
 does not satisfy property BR. One can show this as follows. Using the fact that 
$(\mathbb{D}\setminus\{0\}) \times \mathbb{D}$ is biholomorphic to $\mathbb{H}$   
one can compute that the Bergman kernel of $\mathbb{H}$  
(see, for instance, \cite{Edholm16,EdholmMcNeal16})
\[K_{\mathbb{H}}(w,z) 
 =\frac{w_1\zb_1}{\pi^2(w_1\zb_1-w_2\zb_2)^2(1-w_1\zb_1)^2}.\] 
Then for $0<\delta,\ep<1$ we have 
\[\frac{|K_{\mathbb{H}}((\ep,0),(\delta,0))|}{K_{\mathbb{H}}((\delta,0),(\delta,0))} 
=\frac{\delta(1-\delta^2)^2}{\ep(1-\delta\ep)^2}.\]
Now if we fix $\delta$ and let $\ep\to 0$ we see that the ratio 
$|K_{\mathbb{H}}((\ep,0),(\delta,0))|/K_{\mathbb{H}}((\delta,0),(\delta,0))\to \infty$. 
Hence $\mathbb{H}$ does not satisfy property BR. 

Domains satisfying property BR include the domains with sharp $\mathcal{B}$-type 
Bergman kernels (see Lemma \ref{LemKernelEst}) and their cross products such as the 
polydisc (see Corollary \ref{CorKernelEst}). Since it is technical, we will not give  a precise 
definition of domains with sharp $\mathcal{B}$-type Bergman kernels and we refer the reader to 
\cite{KhanhLiuThuc19}. However, we note that domains with sharp $\mathcal{B}$-type 
Bergman kernels form a large class of domains in $\C^n$  including: strongly 
pseudoconvex domains (such as the unit ball), pseudoconvex finite type domains 
in $\C^2$, finite type convex domains,  decoupled domains of finite type, and 
pseudoconvex domains of finite type whose Levi-forms have only one degenerate 
eigenvalue or comparable eigenvalues. Domains with sharp $\mathcal{B}$-type 
Bergman kernels were first introduced by McNeal and Stein in \cite{McNealStein94} 
in which they prove that if the Bergman kernel of a finite type convex domain is of 
$\mathcal{B}$-type, then the absolute Bergman projection, defined below, is bounded 
on $L^p$ for $1<p<\infty.$ 

The absolute Bergman projection $P^+_{\D}$ of $\D$ is defined as 
\[P^+_{\D}f(z)=\int_{\D}|K_{\D}(z,w)|f(w)dV(w).\]

In the proposition below, we show that $L^p$ regularity of the absolute Bergman projection 
implies $L^p$ regularity of the Berezin transform for domains with property BR. The main 
reason the absolute Bergman projection has been studied is the following simple fact: 
$L^p$-regularity of the absolute Bergman projection implies the $L^p$-regularity of 
the Bergman projection. For an incomplete list we refer the reader to 
\cite{PhongStein77,McNealStein94,McNeal94b,LanzaniStein12,EdholmMcNeal16,Huo18,
HuoWagnerWick21} and references therein. 

\begin{proposition}\label{Prop1}
Let $\D$ be a domain in $\C^n$ satisfying property BR and  the absolute 
Bergman projection $P^+_{\D}:L^{p_0}(\D)\to L^{p_0}(\D)$ is bounded for some 
$1< p_0<\infty$. Then $B_{\D}:L^p(\D)\to L^p(\D)$ is bounded for all 
$p_0\leq p\leq \infty$ and 
\[\|B_{\D}\|_{L^{p_0}} 
\leq \sup\left\{\frac{|K_{\D}(w,z)|}{K_{\D}(z,z)}:z,w\in \D \right\} \|P^+_{\D}\|_{L^{p_0}}.\]
\end{proposition} 

The absolute Bergman projections are $L^p$ regular on domains whose Bergman 
kernels are of sharp $\mathcal{B}$-type (see Lemma \ref{LemP+}, below). 
On products of such domains we have the following result.  

\begin{theorem}\label{Thm1}
Let $\D=\D_1\times \cdots \times \D_m$ be a finite product of bounded  
 domains whose Bergman kernels are of sharp $\mathcal{B}$-type. Then the 
Berezin transform $B_{\D}:L^p(\D)\to L^p(\D)$ is bounded for all 
$1< p\leq \infty$. Furthermore, 
\[\|B_{\D}\|_{L^p}\leq 
\sup\left\{\frac{|K_{\D}(w,z)|}{K_{\D}(z,z)}:z,w\in \D \right\} 
\|P^+_{\D_1}\|_{L^p}\cdots \|P^+_{\D_m}\|_{L^p} \]
for  $1<p<\infty$ and where $P^+_{\D_j}$ is the absolute Bergman projection 
of $\D_j$ for $1\leq j\leq m$. 
\end{theorem}

Our second result is about the $L^2$ irregularity of the Berezin transform 
on the Hartogs triangle. 
\begin{theorem}\label{ThmHartogs} 
Let $\Ha=\{(z_1,z_2)\in \C^2:|z_2|<|z_1|<1\}$ be the Hartogs triangle. Then the Berezin 
transform $B_{\mathbb{H}}$ is not a bounded mapping on $L^2(\mathbb{H})$. 
\end{theorem}
\section*{Proofs of Proposition \ref{Prop1} and Theorem \ref{Thm1}}
We start this section with the proof of Proposition \ref{Prop1}. 

\begin{proof}[Proof of Proposition \ref{Prop1}]
First, as mentioned in the introduction, $B_{\D}:L^{\infty}(\D)\to L^{\infty}(\D)$ is 
bounded and $\|B_{\D}\|_{L^{\infty}}\leq 1$.
 
Let  $C_{\D}= \sup\left\{\frac{|K_{\D}(w,z)|}{K_{\D}(z,z)}:z,w\in \D \right\}<\infty$ 
and  $\phi\in L^{\infty}(\D)$. Then
\begin{align*}
 |B_{\D}\phi(z)| \leq &  \int \frac{|K_{\D}(w,z)|^2}{K_{\D}(z,z)}|\phi(w)| dV(w)\\
\leq & C_{\D} \int |K_{\D}(w,z)||\phi(w)| dV(w)\\
=& C_{\D}P^+_{\D}|\phi|(z).
\end{align*}
Then
\[\|B_{\D}\phi\|_{L^{p_0}}\leq C_{\D}\|P^+_{\D}|\phi|\|_{L^{p_0}} 
\leq C_{\D}\|P^+_{\D}\|_{L^{p_0}} \|\phi\|_{L^{p_0}}.\]
Therefore, using the fact that measurable bounded compactly supported functions 
are dense in $L^{p_0}(\D)$, we conclude that  
\[\|B_{\D}\|_{L^{p_0}}\leq C_{\D}\|P^+_{\D}\|_{L^{p_0}}.\]
Then Riesz-Thorin Interpolation Theorem (see, for instance, \cite[Section 6.5]{FollandBook}) 
implies that $B_{\D}:L^p(\D)\to L^p(\D)$ is bounded for all $p_0\leq p\leq \infty$.  
\end{proof}

\begin{remark} \label{Rmk1} 
We note that the Berezin transform is not a self-adjoint operator, in general. 
Indeed one can compute that 
\begin{align*}
\langle B_{\D}\phi,\psi \rangle 
=& \int \int \phi(w)\frac{|K_{\D}(w,z)|^2}{K_{\D}(z,z)}\overline{\psi(z)}dV(w)dV(z)\\
=& \int \phi(w)\overline{K_{\D}(w,w) 
	\int \frac{|K_{\D}(z,w)|^2}{K_{\D}(w,w)}\frac{\psi(z)}{K_{\D}(z,z)} dV(z)}dV(w)\\ 
=&\langle \phi,M_{K_{\D}}B_{\D}M_{\frac{1}{K_{\D}}}\psi \rangle 
\end{align*} 
for $\phi,\psi\in C^{\infty}_0(\D)$ where $M$ denotes the multiplication operator. 
That is $B_{\D}^*= M_{K_{\D}}B_{\D}M_{\frac{1}{K_{\D}}}$. Then, on the unit disc we have 
$B_{\mathbb{D}} 1=1$ while 
\[B_{\mathbb{D}}^*1(0) = \int_{\mathbb{D}}\frac{|K_{\mathbb{D}}(0,w)|^2dV(w)}{K_{\mathbb{D}}(w,w)} 
= \frac{1}{\pi}\int_{\mathbb{D}}(1-|w|^2)^2dV(w)=2\int_0^1(1-r^2)^2rdr= \frac{1}{3}.\] 
Hence $B_{\mathbb{D}}^* \neq B_{\mathbb{D}}$. 

We note that, since $B_{\D}$ is not necessarily self-adjoint, we cannot use $L^p$ regularity 
of $B_{\D}$ for $p_0\leq p\leq \infty$ in the proof of Proposition \ref{Prop1} to conclude 
$L^p$ regularity for $1<p<p_0/(p_0-1)$. 
\end{remark}

We will skip the proof of the following lemma as it is simply a consequence 
of the fact that the Bergman kernel of the product of domains is the product of the 
Bergman kernel of each factor. 
\begin{lemma}\label{LemProd}
Let $\D=\D_1\times\cdots\times \D_m$ be a product domain such that  each $\D_j$ is a 
domain satisfying  property BR. Then $\D$ satisfies property BR. 
\end{lemma}
\begin{lemma}\label{LemProdP+}
Let $\D=\D_1\times\cdots\times \D_m$ be a product domain such that  each 
$P^+_{\D_j}:L^p(\D_j)\to L^p(\D_j)$ is bounded for some $1\leq p<\infty$. Then 
$P^+_{\D}:L^p(\D)\to L^p(\D)$ is bounded and 
\[\|P^+_{\D}\|_{L^p} = \|P^+_{\D_1}\|_{L^p(\D_1)} \cdots \|P^+_{\D_m}\|_{L^p(\D_m)}.\] 
\end{lemma}
\begin{proof}
For $1>\ep>0$ we choose $\phi_{j,\ep}\in L^p(\D_j)$ 
such that $\|\phi_j\|_{L^p(\D_j)}=1$ and 
\[\|P^+_{\D_j}\phi_{j,\ep}\|_{L^p(\D_j)}\geq \|P^+_{\D_j}\|_{L^p}-\ep\] 
for $j=1,\ldots, m$. Then for $\phi_{\ep}=\phi_{1,\ep}\cdots \phi_{m,\ep}$ we have 
$\|\phi_{\ep}\|_{L^p(\D)}=1$ and 
\begin{align*}
(\|P^+_{\D_1}\|_{L^p(\D_1)} -\ep)\cdots (\|P^+_{\D_m}\|_{L^p(\D_m)}-\ep)  
\leq& \|P^+_{\D_1}\phi_{1,\ep}\|_{L^p(\D_1)} \cdots \|P^+_{\D_m}\phi_{m,\ep}\|_{L^p(\D_m)} \\
=&\|P^+_{\D}\phi_{\ep}\|_{L^p(\D)}\\
\leq& \|P^+_{\D}\|_{L^p(\D)}.
\end{align*} 
Then by letting $\ep\to 0$ we get 
\[\|P^+_{\D_1}\|_{L^p(\D_1)} \cdots \|P^+_{\D_m}\|_{L^p(\D_m)}\leq \|P^+_{\D}\|_{L^p(\D)}.\]

To prove the converse, we first assume that $\D=\D_1\times \D_2$. 
Let $\phi\in L^p(\D)$. Then using Fubini's Theorem below we have 
\begin{align*}
&\|P^+_{\D}\phi\|^p_{L^p}\\
&\leq \int_{\D_2} \int_{\D_1}\left(\int_{\D_1} |K_{\D_1}(z_1,w_1)|\int_{\D_2} 
|K_{\D_2}(z_2,w_2)| |\phi(w_1,w_2)|dV(w_2) dV(w_1)\right)^pdV(z_1) dV(z_2)\\
&\leq \|P^+_{\D_1}\|^p_{L^p(\D_1)}\int_{\D_2} \int_{\D_1}\left( \int_{\D_2} 
 |K_{\D_2}(z_2,w_2)||\phi(z_1,w_2)|dV(w_2)\right)^pdV(z_1) dV(z_2)\\
&= \|P^+_{\D_1}\|^p_{L^p(\D_1)}\int_{\D_1} \int_{\D_2}\left( \int_{\D_2} 
 |K_{\D_2}(z_2,w_2)||\phi(z_1,w_2)|dV(w_2)\right)^pdV(z_2) dV(z_1)\\
&\leq\|P^+_{\D_1}\|^p_{L^p(\D_1)}\|P^+_{\D_2}\|^p_{L^p(\D_2)}\int_{\D}|\phi(z)|^pdV(z).
\end{align*} 
Hence $\|P^+_{\D}\|_{L^p} \leq \|P^+_{\D_1}\|_{L^p(\D_1)}\|P^+_{\D_2}\|_{L^p(\D_2)}$. 
Then we use an induction argument to conclude that 
$\|P^+_{\D}\|_{L^p} \leq \|P^+_{\D_1}\|_{L^p(\D_1)}\cdots \|P^+_{\D_m}\|_{L^p(\D_m)}$. 
\end{proof}

The following proposition is an easy consequence of \cite[Proposition 2.4]{KhanhLiuThuc19}. 
We sketch the proof here for the convenience of the reader. 

\begin{lemma}\label{LemP+} 
Let $\D$ be a bounded domain in $\C^n$ whose Bergman kernel is of $\mathcal{B}$-type. 
Then the absolute Bergman projection $P^+_{\D}:L^p(\D)\to L^p(\D)$ is bounded for  
all $1<p<\infty$.
\end{lemma} 
\begin{proof} 
Let $\rho(z)$ denote the distance from $z$ to the boundary of $\D$. Then 
\cite[Proposition 2.4]{KhanhLiuThuc19} implies that  for $0<\ep<1$ there exists $c_{\ep}$ 
such that $P^+_{\D}\rho^{-\ep}\leq c_{\ep}\rho^{-\ep}$ on $\D$. Then a standard argument 
(see, for instance, \cite[pg 184]{McNealStein94}) using the Schur's test 
implies that $P^+_{\D}:L^p(\D)\to L^p(\D)$ is bounded for  all $1<p<\infty$.
\end{proof}

Next we prove the fact that domains whose Bergman kernels are of 
sharp $\mathcal{B}$-type satisfy property BR. 
 
\begin{lemma}\label{LemKernelEst}
Let $\D$ be a bounded domain in $\C^n$ whose Bergman kernel is of 
sharp $\mathcal{B}$-type. Then $\D$ satisfies property BR. 
\end{lemma}
\begin{proof}
We use an argument similar to \cite[(2.5)]{KhanhLiuThuc19} to conclude 
that there exists  $0<C<\infty$ such that 
\[\frac{|K_{\D}(w,z)|}{K_{\D}(z,z)}
	\leq \frac{C|r(z)|^2}{(|r(z)|+|r(w)|)^2}\leq C\]
for all $z,w\in \D$. Therefore, 
$\sup\left\{\frac{|K_{\D}(w,z)|}{K_{\D}(z,z)}:z,w\in \D \right\}<\infty$.
\end{proof}

\begin{corollary}\label{CorKernelEst}
Let $\D=\D_1\times\cdots\times \D_m$ be a product domain such that each $\D_j$ is a 
bounded domain whose Bergman kernel is of sharp $\mathcal{B}$-type. Then 
\begin{itemize}
\item[i.]  $\D$ satisfies property BR,
\item[ii.] the absolute Bergman projection $P^+_{\D}:L^p(\D)\to L^p(\D)$ is bounded for  
all $1<p<\infty$ and 
\begin{align}\label{Eqn3}
\|P^+_{\D}\|_{L^p} = \|P^+_{\D_1}\|_{L^p(\D_1)} \cdots \|P^+_{\D_m}\|_{L^p(\D_m)}.
\end{align}
\end{itemize} 
\end{corollary}

\begin{proof}
To prove i. we use the fact that the Bergman kernel of $\D$ is the product of 
the Bergman kernels of $\D_j$s together with Lemma \ref{LemKernelEst} to show that  
\[\sup\left\{\frac{|K_{\D}(w,z)|}{K_{\D}(z,z)}:z,w\in \D \right\}<\infty.\]
Hence, $\D$ satisfies property BR.  

ii. is simply is a consequence of Lemmas \ref{LemProdP+} and \ref{LemP+}.   
\end{proof}

\begin{proof}[Proof of Theorem \ref{Thm1}] 
The proof is very similar to the proof of Proposition \ref{Prop1} except here 
we use \eqref{Eqn3} in Corollary \ref{CorKernelEst}.
\end{proof}

\section*{Proof of Theorem \ref{ThmHartogs}}
Before we start the proof of Theorem \ref{ThmHartogs}, we note that $P^+_{\Ha}$ is $L^p$ 
regular for $4/3< p< 4$ (see proof of \cite[Theorem 1.2]{EdholmMcNeal16}) and 
\cite[Corollary 3.5]{EdholmMcNeal16}). However, we will show that the Berezin transform 
is not $L^2$ regular on the Bergman space on the Hartogs triangle. Therefore, 
$L^p$-regularity of the absolute Bergman projection is not a sufficient 
condition for the $L^p$-regularity of the Berezin transform. 

\begin{proof}[Proof of Theorem \ref{ThmHartogs}]
The set  $\{z_1^nz_2^m:m\geq 0,n+m\geq -1\}$ forms an orthogonal basis for $A^2(\Ha)$ 
(see, for example, \cite{ChakrabartiZeytuncu16}). Let $f_{\ep}(w)=|w_1|^{-2+2\ep}$ for $\ep>0$. 
Then $f_{\ep}\in L^2(\Ha)$ because when $\ep>0$, 
the following calculation holds  
\begin{align*} 
\|f_{\ep}\|^2_{L^2}=4\pi^2\int_0^1\int_0^{r_1}r_1^{-3+4\ep} r_2 dr_2dr_1  
=2\pi^2\int_0^1r_1^{-1+4\ep}dr_1=\frac{\pi^2}{2\ep}<\infty, 
\end{align*} 
and when $\ep\leq 0$, the integral diverges. 
We can write the Bergman kernel $K_{\Ha}$ in series form as follows
\[K_{\Ha}(w_1,w_2,z_1,z_2)=\sum_{m\geq 0,n+m\geq -1} a_{nm} w_1^nw_2^m\zb_1^n\zb_2^m\]
where 
\[a_{nm}=\frac{1}{\|z_1^nz_2^m\|^2_{L^2}} 
=\frac{1}{4\pi^2 \int_0^1\int_0^{r_1}r_1^{2n+1}r_2^{2m+1}dr_2dr_1}
=\frac{(m+1)(n+m+2)}{\pi^2}\]
for $m\geq 0$ and $n+m+1\geq 0$. Next we compute 
\begin{align*}
\langle f_{\ep}K_{\Ha}(.,z),K_{\Ha}(.,z)\rangle 
=&\sum_{m\geq 0,n+m\geq -1}  |a_{nm}|^2|z_1|^{2n}|z_2|^{2m}4 
\pi^2\int_0^1\int_0^{r_1} r_1^{2n-1+2\ep}r_2^{2m+1}dr_2dr_1\\
=&\sum_{m\geq 0,n+m\geq -1} |a_{nm}|^2|z_1|^{2n}|z_2|^{2m} 
\frac{\pi^2}{(m+1)(n+m+1+\ep)}.
\end{align*}
Using $k=n+m+1$ in the second equality below we get  
\begin{align*}
B_{\Ha}f_{\ep}(z)=&\frac{1}{\pi^2K_{\Ha}(z,z)}
\sum_{m\geq 0,n+m\geq -1}|z_1|^{2n}|z_2|^{2m} 
	\frac{(m+1)(n+m+2)^2}{n+m+1+\ep}\\ 
 =&\frac{1}{\pi^2K_{\Ha}(z,z)} \sum_{m,k=0}^{\infty}|z_1|^{2k-2m-2}|z_2|^{2m} 
	\frac{(m+1)(k+1)^2}{k+\ep}\\
=&\frac{1}{\pi^2|z_1|^2K_{\Ha}(z,z)} \left(\sum_{m=0}^{\infty}(m+1) 
	\frac{|z_2|^{2m}}{|z_1|^{2m}}\right) 
	\left( \sum_{k=0}^{\infty}\frac{(k+1)^2}{k+\ep}|z_1|^{2k}\right). 
\end{align*} 
Next we note that  
\begin{align}\label{Eqn4} 
\frac{1}{\pi^2|z_1|^2K_{\Ha}(z,z)} 
= \frac{(|z_1|^2-|z_2|^2)^2(1-|z_1|^2)^2}{|z_1|^4}
= \left(1-\left|\frac{z_2}{z_1}\right|^2\right)^2(1-|z_1|^2)^2
\end{align}  
and $\frac{1}{(1-x)^2}=\sum_{m=0}^{\infty}(m+1)x^m$ for $|x|<1$. Therefore, 
\[B_{\Ha}f_{\ep}(z)=(1-|z_1|^2)^2\sum_{k=0}^{\infty} 
	\frac{(k+1)^2}{k+\ep}|z_1|^{2k}.\]
Then 
\begin{align*}
\|B_{\Ha}f_{\ep}\|^2_{L^2} 
= & \int_{\Ha}|1-|z_1|^2|^4\left|\sum_{k=0}^{\infty} 
	\frac{(k+1)^2}{k+\ep}|z_1|^{2k}\right|^2dV(z)\\
\geq & \frac{1}{\ep^2} \int_{\Ha}|1-|z_1|^2|^4dV(z).
\end{align*} 
Then, we have $ \|f_{\ep}\|_{L^2} =\pi/\sqrt{2\ep}$ while 
$\|B_{\Ha}f_{\ep}\|_{L^2} \geq  \ep^{-1}\|(1-|z_1|^2)^2\|_{L^2}$ 
for $\ep>0$. Hence 
\[\frac{\|B_{\Ha}f_{\ep}\|_{L^2}}{\|f_{\ep}\|_{L^2}} \geq  \frac{\sqrt{2} \|(1-|z_1|^2)^2\|_{L^2}}{\pi \sqrt{\ep}}\] 
for $\ep>0$. Therefore, $B_{\Ha}$ is unbounded on $L^2(\Ha)$. 
We note that, by interpolation, we also conclude that $B_{\Ha}$ is unbounded on 
$L^p(\Ha)$ for $p\leq 2$.
\end{proof} 
\begin{remark}
Let $f(z)=z_1^{-1}\in A^2(\Ha)$. Then 
\[\langle f, k^{\Ha}_z\rangle=\frac{1}{z_1\sqrt{K_{\Ha}(z,z)}}.\]
Then  \eqref{Eqn4}, for $z_1=1/j$ and $z_2=0$, implies that 
\[|\langle f, k^{\Ha}_{(1/j,0)}\rangle|=\pi (1-j^{-2})\to \pi \text{ as } j\to\infty.\]
Therefore, $\{k^{\Ha}_{(1/j,0)}\}$ does not converge to 0 weakly.
\end{remark}

\section*{Acknowledgement}
We would like to thank the referee for helpful comments that has improved the paper.


\end{document}